\documentclass[11pt]{amsart}
\usepackage{mathpazo}
\usepackage{geometry}              
\usepackage{graphicx}
\usepackage[parfill]{parskip}
\usepackage{amssymb}
\usepackage[all]{xy}
\usepackage{color}
\usepackage{tikz}
\usetikzlibrary{arrows,decorations.markings,fit,matrix,positioning,calc,shapes.symbols,shapes}
\usepackage{dsfont}

\definecolor{dmagenta}{rgb}{.5,0,.5} 
\definecolor{dred}{rgb}{.5,0,0} 
\definecolor{dgreen}{rgb}{0,.5,0} 
\definecolor{dblue}{rgb}{0,0,0.5} 
\definecolor{black}{rgb}{0,0,0} 
\definecolor{vdgreen}{rgb}{0,.3,0} 
\definecolor{vdred}{rgb}{.3,0,0} 
\definecolor{red}{rgb}{1,0,0} 

\newcommand{\hairy}{\mathcal H}  

\newcommand{\ext}{\bigwedge\nolimits}

\DeclareMathOperator{\Tr}{Tr}  
\DeclareMathOperator{\SP}{Sp}  

\DeclareMathOperator{\im}{im}  

\DeclareMathOperator{\GL}{GL}  

\newtheorem{proposition}{Proposition}[section]
\newtheorem{theorem}[proposition]{Theorem}

\theoremstyle{remark}

\theoremstyle{definition}

\newtheorem{remark}[proposition]{Remark}

\newtheoremstyle{red}{3pt}{3pt}{\color{red}}{}{\itshape}{.}{.5em}{}
\theoremstyle{red}

\makeatletter
\def\Ddots{\mathinner{\mkern1mu\raise\p@
\vbox{\kern7\p@\hbox{.}}\mkern2mu
\raise4\p@\hbox{.}\mkern2mu\raise7\p@\hbox{.}\mkern1mu}}
\makeatother

\tikzset{
empty/.style={inner sep=-1pt,minimum size=0mm},
emptyunit/.style={circle,draw=white,fill=white,thick, inner sep=0pt,minimum size=2mm},
emptyantipode/.style={circle,draw=white,fill=white,thick, inner sep=0pt,minimum size=4mm},
operad/.style={circle,draw=red!50,fill=red!20,thick, inner sep=0pt,minimum size=6mm},
hopf/.style={signal, signal to=east, signal from=west,draw=brown!50,fill=brown!20,thick, inner sep=2pt,minimum size=6mm},
lhopf/.style={signal, signal to= west, signal from= east,draw=brown!50,fill=brown!20,thick, inner sep=2pt,minimum size=6mm},
antipode/.style={circle,draw=purple!50,fill=purple!20,thick, inner sep=0pt,minimum size=4mm},
unit/.style={circle,draw=black,fill=white,thick, inner sep=0pt,minimum size=2mm},
break/.style={inner sep=0pt,minimum size=5mm},
block/.style={draw=blue,fill=blue!20,thick,inner sep=10pt}
outer/.style={}
}

\title[The ES-trace detects all top-level partitions]{Addendum to ``The Johnson cokernel and the Enomoto-Satoh invariant":
The ES-trace detects all top-level partitions}

\author{Jim Conant}

\begin{document}
\begin{abstract}
The degree $d$ part of the cokernel $\mathsf C_d$ of the Johnson homomorphism decomposes into irreducible $\SP$-modules indexed by partitions of $d-2r$ for $r\geq 0$:  $$\mathsf C_d\cong \mathsf C_d(d)\oplus \mathsf C_d(d-2)\oplus\cdots.$$ 
In this addendum we calculate $\mathsf{C}_d(d)$ precisely: it is isomorphic to the $\GL(V)$-decomposition of a space of coinvariants $(V^{\otimes d})_{D_{2d}}$, and the isomorphism is induced by Enomoto and Satoh's trace map. This establishes Conjecture 7.2 of the paper ``The Johnson Cokernel and the Enomoto-Satoh invariant."
\end{abstract}
\maketitle
\section{Introduction}
We refer to \cite{C} for all definitions.

Let $\mathsf C_d$ be the degree $d$ part of the Johnson cokernel. 
It is known that $\mathsf C_d\cong  \mathsf C_d(d)\oplus \mathsf C_d(d-2)\oplus\cdots,$ where $\mathsf C_d(k)$ is a direct sum of irreducible $\SP$-modules indexed by partitions of the integer $k$.

It is the purpose of this addendum to prove the following theorem, which establishes Conjecture 7.2 of \cite{C}. We slightly modify our notation and use the letter $V$ to denote our symplectic vector space, rather than $H$ as in \cite{C}. 
Recall that $V^{\langle d\rangle}\subset V^{\otimes d}$ is the intersection of the kernels of all pairwise contractions and that $\pi\colon V^{\otimes d}\to V^{\langle d\rangle}$ is a natural projection. (See section 5 of \cite{C}). Recall also the Enomoto-Satoh trace map (cf. Theorem 4.2 of \cite{C}) $$\Tr^{ES}\colon \mathsf C_d\to [V^{\otimes d}]_{D_{2d}}.$$ 

\begin{theorem}\label{thm:main}
$\pi\circ \Tr^{ES}$ induces an isomorphism
$\mathsf C_d(d)\cong [V^{\langle d\rangle}]_{D_{2d}}.$
\end{theorem}
\begin{remark}
The $\SP$-decomposition of $ [V^{\langle d\rangle}]_{D_{2d}}$ is isomorphic to the $\GL(V)$-decomposition of $ [V^{\otimes d}]_{D_{2d}}$, which can be analyzed using standard representation theory methods. See section 7 of \cite{C}.
\end{remark}


\section{Proof of the main theorem}
%
Recall from \cite{C} that $C_k\mathcal H(V)$ denotes the part of the hairy Lie graph complex spanned by connected graphs with $k$ vertices of valence $\geq 3$ colored by elements of the Lie operad, with some number of hairs labeled by elements of the vector space $V$. We also define the \emph{Lie degree} of a hairy Lie graph as the sum of the degrees of the Lie operad elements at the vertices of the graph, and where the degree of an element of the Lie operad is the number of trivalent vertices in the tree. Let us denote the set of elements in $C_k\mathcal H(V)$ which are of Lie degree $d$ as $\mathcal H_{k,d}(V)$. Let $\mathcal H^{\mathrm{ord}}_{k,d}(V)$ be defined similarly, but with a specified global order on the $\mathsf{Lie}$-colored vertices.  It is clear that $\mathcal H^{\mathrm{ord}}_{1,d}=\mathcal H_{1,d}$.

 Let $\beta\colon \mathcal H^{\mathrm{ord}}\to \mathcal H^{\mathrm{ord}}$ be defined by summing over all contractions between vertices $1$ and $2$, the resulting graphs being given the obvious ordering. 
  Notice that $\mathcal H^{\mathrm{ord}}_{d,d}$ must have all vertices colored by tripods. 
  We have
 $$\beta^{d-1}\colon \mathcal H^{\mathrm{ord}}_{d,d}(V)\to \mathcal H_{1,d}(V).$$
See Figure~\ref{fig:fig1} for examples.

 The spaces $\mathcal H_{k,d}(V)$ and $\mathcal H^{\mathrm{ord}}_{k,d}(V)$ have versions $\mathcal H_{k,d}\langle V\rangle $ and $\mathcal H^{\mathrm{ord}}_{k,d}\langle V\rangle $ in which the labeling coefficients are assumed to be in $V^{\langle s\rangle}$, where $s$ is the number of hairs. (See section 5 of \cite{C}.)

The next theorem gives an exact graphical characterization of the (stable) Johnson cokernel.
%
%
%

\begin{theorem}
If $\dim(V)\gg d$, we have
$$\mathsf C_d(V)\cong \hairy_{1,d}\langle V\rangle/\beta^{d-1}(\mathcal H^{\mathrm{ord}}_{d,d}\langle V\rangle).$$
\end{theorem}
\begin{proof}
The Johnson cokernel, by Hain's theorem (cf. introduction to \cite{C}), can be seen as the quotient of $\mathcal T_d(V)$ by iterated brackets of tripods. I.e. let $br\colon \mathcal T_1(V)^{\otimes d}\to \mathcal T_d(V)$ be the map $t_1\otimes\cdots\otimes t_d\mapsto [[[t_1,t_2],t_3],\cdots,t_d]$. Then stably $C_d(V)=\mathcal T_d(V)/\im(br)$.

We consider the following commutative diagram:
$$\xymatrix{
\mathcal T_1(V)^{\otimes d}\ar@{->}^{\Tr}[r]\ar@{->}_{br}[d]&\mathcal H^{\mathrm{ord}}_{d,d}(V)\ar@{->}[d]^{\beta^{d-1}}\ar@{->}^{\pi}[r]&\mathcal H^{\mathrm{ord}}_{d,d}\langle V\rangle\ar@{->}_{\beta^{d-1}}[d]\\
\mathcal T_d(V)\ar@{->}^{\Tr}[r]&\mathcal H_{1,d}(V)\ar@{->}^{\pi}[r]&\mathcal H_{1,d}\langle V\rangle
}$$
The trace map $\mathcal T_1(V)^{\otimes d}\to \mathcal H^{\mathrm{ord}}_{d,d}(V)$ is defined on tensor products of tripods, by taking their disjoint union and adding edges in all possible ways. The tripods inherit an order from the tensor product. 

The commutativity of the left square is related to the fact that $\Tr$ is a chain map. Let $t_1\otimes \cdots \otimes t_d$ be a tensor product of tripods. Applying $br$ to this tensor product sums over joining $t_1$ to $t_2$ along a (solid) edge in all ways, then joining $t_3$ to the result and so forth. Taking the trace now adds some additional dotted edges connecting some of the hairs. On the other hand, taking the trace of  $t_1\otimes \cdots \otimes t_d$ means adding several dotted edges to some of the hairs. Applying $\beta$ now joins $t_1$ to $t_2$ by converting dotted edges joining them to solid edges, i.e. contracting the first two trees in all possible ways in the graph. Applying $\beta$ again contracts the now-joined $t_1$ and $t_2$ with the third tree etc. So $\beta^{d-1}(\Tr(t_1\otimes\cdots\otimes t_d))$ corresponds to 
taking an iterated bracket of $t_1\otimes\cdots\otimes t_d$ followed by adding some external edges, the same result as going the other way around the square.

The commutativity of the right square follows since $\pi$ operates on the labeling coefficients, which are not involved in the definition of $\beta$.

The composition $\pi\circ\Tr^C$ along the bottom is an isomorphism by Theorem 5.4 of \cite{C} which quotes Corollary 3.7 of \cite{CKV-HHGH}. In fact the top composition $\pi\circ\Tr^C$ is also an isomorphism, which follows from the same method of proof as in \cite{CKV-HHGH}. The only difference is that instead of defining $\Tr$ on $\ext^d \mathcal T_1(V)$, which yields an order on the vertices of the image graph up to even permutation, we instead define it on $\mathcal T_1(V)^{\otimes d}$, yielding a fixed order on the vertices. However, none of the arguments used in the proof of Corollary 3.7 of \cite{CKV-HHGH} is affected by the presence of a global order. 

Thus $$\mathsf{C}_d(V)\cong \mathcal H_{1,d}\langle V\rangle/\pi\Tr^C(\im br)= \hairy_{1,d}\langle V\rangle/\beta^{d-1}(\mathcal H^{\mathrm{ord}}_{d,d}\langle V\rangle).$$
\end{proof}

%
\begin{proof}[Proof of Theorem~\ref{thm:main}]
The top level partitions correspond to rank $1$ graphs.
Let $\mathcal H_{1,d,r}(V)$ be the subspace of $\mathcal H_{1,d}(V)$ spanned by graphs of rank $r$.
  We want to show that $$\mathcal H_{1,d,1}(V)/\im(\beta^{d-1})\cong (V^{\otimes s})_{D_{2d}}.$$ This latter space can be pictured as a loop with hairs, modulo the relation that one can slide a hair across the edge as in Theorem 4.1 of \cite{C}. Numbering the vertices of the loop consecutively, we see that $\beta^{d-1}$ applied to such a graph gives this slide relation, as in Figure~\ref{fig:fig1} (1). Thus $\mathcal H_{1,d,1}/\im(\beta^{d-1})$ is a quotient of $(V^{\otimes d})_{D_{2d}}$. Now consider $\beta^{d-1}$ applied to an arbitrary rank $1$ graph in $\mathcal H^{\mathrm{ord}}_{d,d}( V)$. It will either be $0$, or all but one of the applications of $\beta$ will correspond to a unique contraction.  This one exceptional case will then involve two trees $a$ and $b$ connected by a pair of edges as in Figure~\ref{fig:fig1} (3). Applying $\beta$ iteratively fills in the dotted edges in these trees to trees $\hat{a}$ and $\hat{b}$ respectively. This maps to $\hat{a}\hat{b}-\hat{b}\hat{a} \in (V^{\otimes d})_{D_{2d}}$ and is trivial there, since the two terms differ by a cyclic permutation.
\end{proof}
\begin{figure}
\begin{multline}
\begin{minipage}{5cm}
\resizebox{5cm}{!}{
\begin{tikzpicture}
\coordinate(a) at (0,.5);
\coordinate(b) at (1.5,.5);
\coordinate(c) at (3,.5);
\node[at=(a)] (A){$v_1$};
\node[at=(b)] (B){$v_2$};
\node[at=(c)] (C){$v_3$};
\draw[thick] (-.5,0) to (.5,0);
\draw[thick] (0,0) to (A);
\draw[thick] (1,0) to (2,0);
\draw[thick] (1.5,0) to (B);
\draw[thick] (2.5,0) to (3.5,0);
\draw[thick] (3,0) to (C);
\begin{scope}[decoration={markings,mark = at position 0.5 with {\arrow{stealth}}}]
\draw[densely dashed, postaction=decorate] (.5,0) to[densely dashed, postaction=decorate](1,0);
\draw[densely dashed, postaction=decorate] (2,0) to[densely dashed, postaction=decorate](2.5,0);
\draw[densely dashed, postaction=decorate] (3.5,0) to[densely dashed, postaction=decorate,out=-45,in=225](-.5,0);
\end{scope}
\coordinate(d) at (0,-.3);
\node[at=(d),shape=circle,draw,inner sep=2pt]{\tiny{$1$}};
\coordinate(e) at (1.5,-.3);
\node[at=(e),shape=circle,draw,inner sep=2pt]{\tiny{$2$}};
\coordinate(f) at (3,-.3);
\node[at=(f),shape=circle,draw,inner sep=2pt]{\tiny{$3$}};
\end{tikzpicture}}
\end{minipage}\!\!\!\!\!\!\!\!\!\!\!\!
\overset{\beta}{\mapsto }\!\!\!\!\!\!\!\!\!\!\!\!
\begin{minipage}{5cm}
\resizebox{5cm}{!}{
\begin{tikzpicture}
\coordinate(a) at (0,.5);
\coordinate(b) at (1.5,.5);
\coordinate(c) at (3,.5);
\node[at=(a)] (A){$v_1$};
\node[at=(b)] (B){$v_2$};
\node[at=(c)] (C){$v_3$};
\draw[thick] (-.5,0) to (.5,0);
\draw[thick] (0,0) to (A);
\draw[thick] (1,0) to (2,0);
\draw[thick] (1.5,0) to (B);
\draw[thick] (2.5,0) to (3.5,0);
\draw[thick] (3,0) to (C);
\draw[thick] (.5,0) to (1,0);
\begin{scope}[decoration={markings,mark = at position 0.5 with {\arrow{stealth}}}]
\draw[densely dashed, postaction=decorate] (2,0) to[densely dashed, postaction=decorate](2.5,0);
\draw[densely dashed, postaction=decorate] (3.5,0) to[densely dashed, postaction=decorate,out=-45,in=225](-.5,0);
\end{scope}
\coordinate(e) at (.75,-.3);
\node[at=(e),shape=circle,draw,inner sep=2pt]{\tiny{$1$}};
\coordinate(f) at (3,-.3);
\node[at=(f),shape=circle,draw,inner sep=2pt]{\tiny{$2$}};
\end{tikzpicture}}
\end{minipage}
\!\!\!\!\!\!\!\!\!\!\!\!
\overset{\beta}{\mapsto}
\!\!\!\!\!\!\!\!\!\!\!\!\\
\begin{minipage}{5cm}
\resizebox{5cm}{!}{
\begin{tikzpicture}
\coordinate(a) at (0,.5);
\coordinate(b) at (1.5,.5);
\coordinate(c) at (3,.5);
\node[at=(a)] (A){$v_1$};
\node[at=(b)] (B){$v_2$};
\node[at=(c)] (C){$v_3$};
\draw[thick] (-.5,0) to (.5,0);
\draw[thick] (0,0) to (A);
\draw[thick] (1,0) to (2,0);
\draw[thick] (1.5,0) to (B);
\draw[thick] (2.5,0) to (3.5,0);
\draw[thick] (3,0) to (C);
\draw[thick] (.5,0) to (1,0);
\draw[thick] (2,0) to (2.5,0);
\begin{scope}[decoration={markings,mark = at position 0.5 with {\arrow{stealth}}}]
\draw[densely dashed, postaction=decorate] (3.5,0) to[densely dashed, postaction=decorate,out=-45,in=225](-.5,0);
\end{scope}
\end{tikzpicture}}
\end{minipage}
\!\!\!\!\!\!\!\!\!\!\!\!
-
\!\!\!\!\!\!\!\!\!\!\!\!
\begin{minipage}{5cm}
\resizebox{5cm}{!}{
\begin{tikzpicture}
\coordinate(a) at (0,.5);
\coordinate(b) at (1.5,.5);
\coordinate(c) at (3,.5);
\node[at=(a)] (A){$v_1$};
\node[at=(b)] (B){$v_2$};
\node[at=(c)] (C){$v_3$};
\draw[thick] (-.5,0) to (.5,0);
\draw[thick] (0,0) to (A);
\draw[thick] (1,0) to (2,0);
\draw[thick] (1.5,0) to (B);
\draw[thick] (2.5,0) to (3.5,0);
\draw[thick] (3,0) to (C);
\draw[thick] (.5,0) to (1,0);
\draw[thick] (3.5,0) to[out=-45,in=225](-.5,0);
\begin{scope}[decoration={markings,mark = at position 0.5 with {\arrow{stealth}}}]
\draw[densely dashed, postaction=decorate] (2,0) to[densely dashed, postaction=decorate](2.5,0);
\end{scope}
\end{tikzpicture}}
\end{minipage}
\!\!\!\!\!\!\!\!\!\!\!\!
\leftrightarrow v_1v_2v_3-v_3v_1v_2
\end{multline}
\begin{equation}
\begin{minipage}{5cm}
\resizebox{5cm}{!}{
\begin{tikzpicture}
\coordinate(a) at (0,.5);
\coordinate(b) at (1.5,.5);
\coordinate(c1) at (2.5,2);
\coordinate(c2) at (3.5,2);
\node[at=(a)] (A){$v_1$};
\node[at=(b)] (B){$v_2$};
\node[at=(c1)] (C1) {$v_3$};
\node[at=(c2)] (C2) {$v_4$};
\draw[thick] (-.5,0) to (.5,0);
\draw[thick] (0,0) to (A);
\draw[thick] (1,0) to (2,0);
\draw[thick] (1.5,0) to (B);
\draw[thick] (2.5,0) to (3.5,0);
\draw[thick] (3,0) to (3,.5);
\draw[thick] (C1) to (3,1.5) to (C2);
\draw[thick] (3,1.5) to (3,1);
\begin{scope}[decoration={markings,mark = at position 0.5 with {\arrow{stealth}}}]
\draw[densely dashed, postaction=decorate] (.5,0) to[densely dashed, postaction=decorate](1,0);
\draw[densely dashed, postaction=decorate] (2,0) to[densely dashed, postaction=decorate](2.5,0);
\draw[densely dashed, postaction=decorate] (3.5,0) to[densely dashed, postaction=decorate,out=-45,in=225](-.5,0);
\draw[densely dashed, postaction=decorate] (3,.5) to[densely dashed, postaction=decorate](3,1);
\end{scope}
\coordinate(d) at (0,-.3);
\node[at=(d),shape=circle,draw,inner sep=2pt]{\tiny{$1$}};
\coordinate(e) at (1.5,-.3);
\node[at=(e),shape=circle,draw,inner sep=2pt]{\tiny{$2$}};
\coordinate(f) at (3,-.3);
\node[at=(f),shape=circle,draw,inner sep=2pt]{\tiny{$4$}};
\coordinate(g) at (3.5,1.5);
\node[at=(g),shape=circle,draw,inner sep=2pt]{\tiny{$3$}};
\end{tikzpicture}}
\end{minipage}
\!\!\!\!\!\!\!\!\!\!\!\!
\overset{\beta}{\mapsto}
\!\!\!\!\!\!\!\!\!\!\!\!
\begin{minipage}{5cm}
\resizebox{5cm}{!}{
\begin{tikzpicture}
\coordinate(a) at (0,.5);
\coordinate(b) at (1.5,.5);
\coordinate(c1) at (2.5,2);
\coordinate(c2) at (3.5,2);
\node[at=(a)] (A){$v_1$};
\node[at=(b)] (B){$v_2$};
\node[at=(c1)] (C1) {$v_3$};
\node[at=(c2)] (C2) {$v_4$};
\draw[thick] (-.5,0) to (.5,0);
\draw[thick] (0,0) to (A);
\draw[thick] (1,0) to (2,0);
\draw[thick] (1.5,0) to (B);
\draw[thick] (2.5,0) to (3.5,0);
\draw[thick] (3,0) to (3,.5);
\draw[thick] (C1) to (3,1.5) to (C2);
\draw[thick] (3,1.5) to (3,1);
\draw[thick] (.5,0) to (1,0);
\begin{scope}[decoration={markings,mark = at position 0.5 with {\arrow{stealth}}}]
\draw[densely dashed, postaction=decorate] (2,0) to[densely dashed, postaction=decorate](2.5,0);
\draw[densely dashed, postaction=decorate] (3.5,0) to[densely dashed, postaction=decorate,out=-45,in=225](-.5,0);
\draw[densely dashed, postaction=decorate] (3,.5) to[densely dashed, postaction=decorate](3,1);
\end{scope}
\coordinate(d) at (.75,-.3);
\node[at=(d),shape=circle,draw,inner sep=2pt]{\tiny{$1$}};
\coordinate(f) at (3,-.3);
\node[at=(f),shape=circle,draw,inner sep=2pt]{\tiny{$3$}};
\coordinate(g) at (3.5,1.5);
\node[at=(g),shape=circle,draw,inner sep=2pt]{\tiny{$2$}};
\end{tikzpicture}}
\end{minipage}
\!\!\!\!\!\!\!\!\!\!\!\!
\overset{\beta}{\mapsto} 0
\end{equation}
\begin{equation}
\cdots\overset{\beta}{\mapsto}
\begin{minipage}{1.5cm}
\begin{tikzpicture}
\coordinate(a) at (0,0);
\node[operad, at=(a)](A){$a$};
\coordinate(b) at (0,-1);
\node[operad, at=(b)](B){$b$};
\begin{scope}[decoration={markings,mark = at position 0.5 with {\arrow{stealth}}}]
\draw[densely dashed, postaction=decorate] (B) to[densely dashed, postaction=decorate, out=180,in=180](A);
\draw[densely dashed, postaction=decorate] (A) to[densely dashed, postaction=decorate, out=0,in=0](B);
\end{scope}
\coordinate(d) at (.5,.25);
\node[at=(d),shape=circle,draw,inner sep=2pt]{\tiny{$1$}};
\coordinate(e) at (.5,-1.25);
\node[at=(e),shape=circle,draw,inner sep=2pt]{\tiny{$2$}};
\end{tikzpicture}
\end{minipage}
\overset{\beta}{\mapsto}
\begin{minipage}{1.5cm}
\begin{tikzpicture}
\coordinate(a) at (0,0);
\node[operad, at=(a)](A){$a$};
\coordinate(b) at (0,-1);
\node[operad, at=(b)](B){$b$};
\draw[thick] (A) to[ out=0,in=0](B);
\begin{scope}[decoration={markings,mark = at position 0.5 with {\arrow{stealth}}}]
\draw[densely dashed, postaction=decorate] (B) to[densely dashed, postaction=decorate, out=180,in=180](A);
\end{scope}
\end{tikzpicture}
\end{minipage}
-
\begin{minipage}{1.5cm}
\begin{tikzpicture}
\coordinate(a) at (0,0);
\node[operad, at=(a)](A){$a$};
\coordinate(b) at (0,-1);
\node[operad, at=(b)](B){$b$};
\draw[thick] (B) to[out=180,in=180](A);
\begin{scope}[decoration={markings,mark = at position 0.5 with {\arrow{stealth}}}]
\draw[densely dashed, postaction=decorate] (A) to[densely dashed, postaction=decorate, out=0,in=0](B);
\end{scope}
\end{tikzpicture}
\end{minipage}
\overset{\beta}{\mapsto}
\cdots\overset{\beta}{\mapsto} \hat{a}\hat{b}-\hat{b}\hat{a}
\end{equation}
\caption{(1):$\beta^2$ applied to an element of $\mathcal H^{\mathrm{ord}}_{3,3}$, yielding an element which is $0$ in $[V^{\otimes 3}]_{D_{6}}$. 
(2) An example where $\beta^3$ applied to an element of $\mathcal H^{\mathrm{ord}}_{4,4}$ is $0$.
(3) The general case of $\beta^{d-1}$ giving a nontrivial combination of graphs. At some point, two trees denoted $a$ and $b$ are numbered $1$ and $2$ and are connected to each other in two different ways.
Thus $\beta$ gives a difference of the two different ways of connecting. Eventually, all of the dotted edges internal to $a$ and $b$ get filled in, yielding trees $\hat{a}$ and $\hat{b}$ which represent elements of the tensor algebra $T(V)$. The result $ \hat{a}\hat{b}-\hat{b}\hat{a}$ is zero modulo cyclic permutation.
} \label{fig:fig1}
\end{figure}
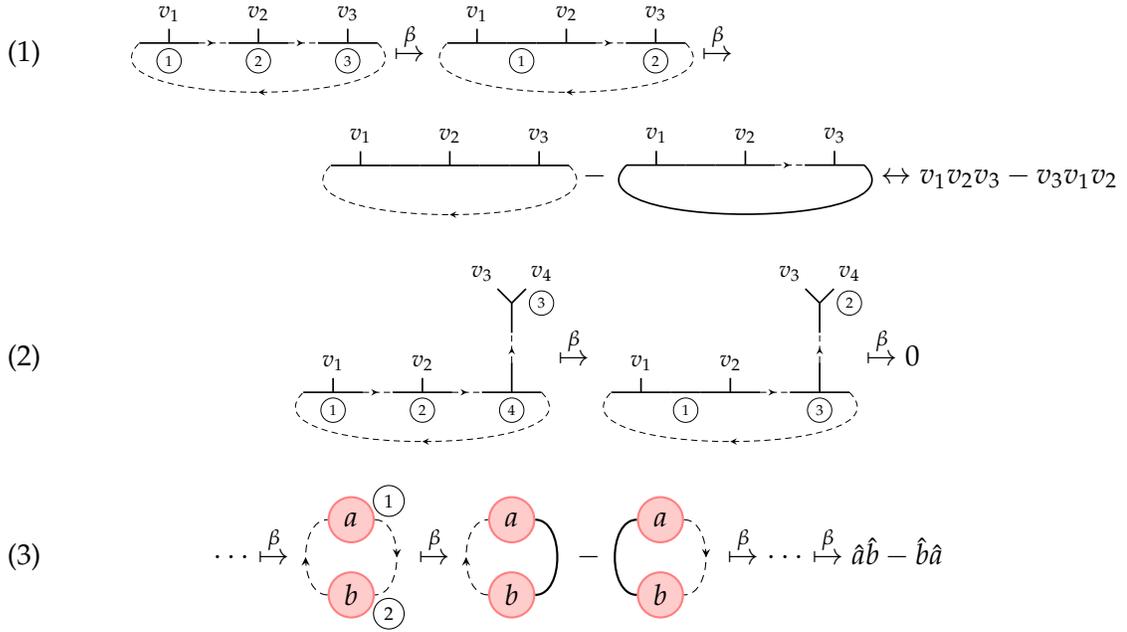



\begin{thebibliography}{9}
\bibitem{C}{\bf J. Conant}, \emph{The Johnson cokernel and the Enomoto-Satoh invariant},  Algebr. Geom. Topol. 15 (2015), no. 2, 801--821.
\bibitem{CKV-HHGH}  {\bf J Conant}, {\bf M Kassabov} and {\bf K Vogtmann}, \emph{ Higher hairy graph homology},  Geom. Dedicata 176 (2015), 345--374.
\end{thebibliography}
\end{document}